\documentclass[11pt]{amsart}

\usepackage{amssymb,amsmath,accents}
\usepackage{bbm,pgf,tikz}
\usetikzlibrary{arrows,automata}
\usepackage[latin1]{inputenc}
\usepackage{a4wide}

\usepackage[colorlinks=true, pdfstartview=FitV, linkcolor=blue, 
citecolor=blue]{hyperref}

\DeclareMathAlphabet{\mymathbb}{U}{bbold}{m}{n}

\newtheorem{theorem}{Theorem}[section]
\newtheorem{prop}[theorem]{Proposition}
\newtheorem{lemma}[theorem]{Lemma}
\newtheorem{fact}[theorem]{Fact}
\newtheorem{coro}[theorem]{Corollary}
\theoremstyle{definition}
\newtheorem{definition}[theorem]{Definition}
\newtheorem{example}[theorem]{Example}
\newtheorem{remark}[theorem]{Remark}

\newcommand{\ts}{\hspace{0.5pt}}
\newcommand{\nts}{\hspace{-0.5pt}}

\newcommand{\RR}{\mathbb{R}\ts}
\newcommand{\CC}{\mathbb{C}\ts}
\newcommand{\NN}{\mathbb{N}}
\newcommand{\ZZ}{\mathbb{Z}}
\newcommand{\PP}{\mathbb{P}}

\newcommand{\cA}{\mathcal{A}}

\newcommand{\cE}{\mathcal{E}}
\newcommand{\cM}{\mathcal{M}}

\newcommand{\ee}{\ts\mathrm{e}}
\newcommand{\omb}{\ts\overline{\nts\omega\nts}\ts}
\newcommand{\ii}{\ts\mathrm{i}}
\newcommand{\dd}{\, \mathrm{d}}
\newcommand{\one}{\mymathbb{1}}
\newcommand{\nix}{\mymathbb{0}}
\newcommand{\ve}{\varepsilon}

\newcommand{\trans}{{\raisebox{1pt}{$\scriptscriptstyle \mathsf{T}$}}}

\newcommand{\exend}{\hfill$\Diamond$}

\newcommand{\diag}{\mathrm{diag}}
\newcommand{\comm}{\mathrm{comm}}
\newcommand{\tr}{\mathrm{tr}}
\newcommand{\Mat}{\mathrm{Mat}}
\newcommand{\bs}[1]{\boldsymbol{#1}}

\newcommand{\defeq}{\mathrel{\mathop:}=}

\newcommand{\myfrac}[2]{\frac{\raisebox{-2pt}{$#1$}}
  {\raisebox{0.5pt}{$#2$}}}

\begin{document}

\title{Embedding of reversible Markov matrices}

\author{Ellen Baake}
\address{Technische Fakult\"at, Universit\"at Bielefeld, 
         Postfach 100131, 33501 Bielefeld, Germany}
\email{ebaake@techfak.uni-bielefeld.de}

\author{Michael Baake}
\address{Fakult\"at f\"ur Mathematik, Universit\"at Bielefeld, 
         Postfach 100131, 33501 Bielefeld, Germany}
\email{mbaake@math.uni-bielefeld.de}
         
\author{Jeremy Sumner}         
\address{School of Natural Sciences, Discipline of Mathematics,
         University of Tasmania,
    \newline \indent PO Box 807, Sandy Bay, TAS 7006, Australia}
\email{Jeremy.Sumner@utas.edu.au}

\begin{abstract} 
  The embeddability of reversible Markov matrices into
  time-homogeneous Markov semigroups is revisited, with
  some focus on simplifications and extensions. In particular,
  we do not demand irreducibility and consider weakly reversible 
  matrices as well as reversible matrices with negative eigenvalues.
\end{abstract}


\maketitle

\section{Introduction}

A \emph{Markov matrix} is a square matrix $M$ with non-negative
entries and unit row sums, while a \emph{Markov generator} $Q$ has
non-negative off-diagonal entries and zero row sums. The latter is
also known as a \emph{rate matrix}, and generates a 
continuous-time semigroup of Markov matrices via 
$\{ \ee^{t \ts Q} : t\geqslant 0 \}$. Let $\cM_d$ denote the convex 
set of all $d \ts {\times} d$ Markov matrices; see \cite{Norris} for 
general background. It is an old and still only partially answered 
question \cite{Elfving, King} to determine which Markov matrices
are \emph{embeddable}, which means that they occur in some Markov
semigroup of the above type. This is the time-homogeneous situation
because $\ee^{t \ts Q}$ is the solution to the matrix-valued initial value
problem $\dot{X} = X Q$ with $X(0) = \one$. This embedding is 
intimately related with the existence of a real matrix logarithm of $M$, 
meaning a real solution $R$ of the equation $M = \ee^R$, which may 
or may not be unique.  The latter, in turn, may or may not be relevant
in a given context. Here, we consider the embedding problem for a 
large but important class of Markov matrices \cite{LP} as follows.

A Markov matrix $M \in \cM_d$ is called \emph{reversible} if there
exists a strictly positive probability vector
$\bs{p} = (p^{}_{1}, \ldots , p^{}_{d} )$ such that the detailed
balance equations $p_i M_{ij} = p_j M_{ji}$ hold for all
\mbox{$1 \leqslant i,j \leqslant d$}; see \cite[Ch.~1]{Kelly} for
background. In this case, $\bs{p}$ must satisfy $\bs{p} M = \bs{p}$
and thus be a left eigenvector of $M$ for the eigenvalue $1$, that is,
$\bs{p}$ is an equilibrium vector for $M$.  Reversible Markov matrices
play an important role both in the theory and in many applications of
Markov chains \cite{LP}; this is due to the invariance under time reversal, 
and the fact that the detailed balance property allows for a
straightforward and explicit calculation of $\bs{p}$.
One concrete motivation for considering reversible Markov chains 
comes from the study of phylogenetics, which encompasses the inference 
of biological evolutionary history from present-day molecular sequence 
data. Here, the most common approach to statistical parameter 
inference works with a hierarchy of reversible Markov models with the 
most general case at the top; see any standard reference such as 
\cite[Ch.~1]{Yang} for further details.  

Questions of embeddability are particularly important in applications
because data are usually sampled at discrete time points, but the
underlying phenomena often operate in continuous time. When working
with Markov matrices, special attention should thus be paid to their
embeddability; however, this has sometimes been overlooked. For
example, in bioinformatics, one of the famous Dayhoff (or PAM) 
matrices \cite{Day-1,Day-2}, a family of reversible Markov matrices that
describe the evolution of amino acids and are routinely used as
scoring matrices for protein sequence alignment (see
\cite[Ch.6.5]{EG} for review), has turned out to not be embeddable
\cite{KG} --- although molecular evolution undoubtedly acts in
continuous time.

Since irreducible Markov matrices have a unique eigenvector
$\bs{p}>\bs{0}$ by the Perron--Frobenius theorem, it is natural to
consider the embedding problem first for irreducible reversible
matrices. This case has been analysed in \cite{Chen} under the further
restriction to reversible generators, where a fairly complete answer
was given, which can still be reformulated in a more algebraic manner
and then be simplified. Inevitably, this only covers reversible Markov 
matrices with positive spectrum.  However, as we shall see below, there 
are reversible Markov matrices with (pairs of) negative eigenvalues, which
calls for some extension.  Also, there is no need to restrict to
irreducible matrices from the beginning, so we will drop this
restriction, too.
\smallskip

The paper is organised as follows. In Section~\ref{sec:prelim}, we 
provide some definitions and preparatory results, mainly from a linear 
algebra perspective.  Then, in Section~\ref{sec:main}, we focus on
reversible Markov matrices and their embedding, simplifying the
approach of \cite{Chen} and extending the discussion to
reducible, weakly reversible matrices, and to cases with 
negative eigenvalues.

In Appendix \ref{sec:app}, we give an interpretation of one embeddable
class of Markov matrices (for $d=3$) with a pair of negative
eigenvalues in terms of cyclic processes, which provides a simpler and
independent approach to this class of matrices.

\section{Preliminaries and preparatory results}\label{sec:prelim}

The real $d\ts{\times}d$ matrices are denoted as $\Mat(d,\RR)$, and
the spectrum of a matrix $B$ is written as $\sigma (B)$, which is
understood as the set or multi-set of eigenvalues, possibly \emph{including}
multiplicities.  When the eigenvalues are distinct, the spectrum is
called \emph{simple}. Further, we speak of \emph{real} or (strictly)
\emph{positive} spectrum when $\sigma(B)\subset \RR$ or
$\sigma(B) \subset \RR_{+}$, respectively, with obvious meaning of 
the inclusion relation in the multi-set case.

In view of later applications to Markov matrices and generators, we
mainly work with row vectors in $\RR^d$, denoted by
$\bs{x} = (x^{}_{1}, \ldots , x^{}_{d})$, while we use
$\bs{x}^{\trans}$ for column vectors. We call $\bs{x}$ strictly
positive, written as $\bs{x} > \bs{0}$, when $x_i >0$ holds for all
$1\leqslant i\leqslant d$.

\begin{definition}\label{def:p-pair}
  Let\/ $\bs{p} = (p^{}_{1}, \ldots , p^{}_{d})$ be a strictly
  positive probability vector, so\/ $p_i > 0$ for all\/
  $1\leqslant i \leqslant d$ and\/ $\sum_{i=1}^{d} p_i = 1$. Then, two
  matrices\/ $A, B \in \Mat(d, \RR)$ form a \emph{$\bs{p}$-balanced
    pair} if the exchange relations\/ $p_i A_{ij} = p_j B_{ji}$ hold
  for all\/ $1\leqslant i,j \leqslant d$.
\end{definition}

No specific relation between $\bs{p}>\bs{0}$ and the matrices $A$ and
$B$ is assumed, but it is easy to verify that $\bs{p}$ is a left
eigenvector of $A$ (of $B$) with eigenvalue $\lambda$ if and only if
all row sums of $B$ (of $A$) are $\lambda$, equivalently if and only
if $\bs{1}^{\!\trans} = (1, \ldots , 1)^{\nts\trans}$ is a right
eigenvector of $B$ (of $A$) with eigenvalue $\lambda$. The exchange
relations can be written in matrix form as
\begin{equation}\label{eq:T-def}
  B \, = \, D^{-1}\nts A^{\trans} \ts D \qquad \text{with} \quad
  D \, = \, D_{\bs{p}} \, \defeq \, \diag ( p^{}_{1}, \ldots , p^{}_{d} ) \ts .
\end{equation}
A motivation for this particular form will follow shortly.  Note that
$D$ possesses a unique matrix square root with positive entries on the
diagonal, as does $D^{-1}$. Clearly, \eqref{eq:T-def} is equivalent with
$D B = A^{\trans} D$, which can still be used when
$\bs{p} \geqslant \bs{0}$ and $D$ is no longer invertible.

\begin{remark}\label{rem:involution}
  When $\bs{p} > \bs{0}$, the mapping
  $A \mapsto \widetilde{A} \defeq D^{-1}\nts A^{\trans} \ts D$ 
  defines an involution on $\Mat (d, \RR)$ that can be seen 
  as the adjoint map for the right action of $A$ in the bilinear 
  form $\sum_{i=1}^{d} x^{}_i \ts p^{}_i \ts y^{}_i$, in complete
  correspondence with the transpose of a matrix being the adjoint map
  for the standard bilinear form $\sum_{i=1}^{d} x^{}_i \ts
  y^{}_i$. The importance of the new involution will become clear in
  the context of reversibility for Markov and related processes.
  
  In the context of stationary Markov chains, both in discrete and
  in continuous time, this involution is induced by the time reversal of
  the chain; see \cite[Ch.~1.9]{Norris} or \cite[Ch.~1.2]{Kelly} for
  background. Our definition is algebraic in nature and applies
  to a more general setting, which will pay off later in our analysis.
  \exend
\end{remark}
 
Of particular interest in Definition~\ref{def:p-pair} is the case
$B=A$, where the exchange relations are called the equations
of \emph{detailed balance}. We use this term in slightly greater
generality than usual for ease of accommodating Markov matrices, 
generators and some of their extensions under a common roof. 

\begin{definition}\label{def:reversible}
  A matrix\/ $B\in\Mat (d,\RR)$ is called \emph{reversible for the
    probability vector}\/ $\bs{p}>\bs{0}$, or\/
  $\bs{p}\ts\ts$-reversible for short, if it satisfies the detailed
  balance equations\/ $p_i B_{ij} = p_j B_{ji}$ for all\/
  $1\leqslant i,j \leqslant d$, or, equivalently, if\/
  $B = D^{-1} \nts B^{\trans} \ts D$ with the diagonal matrix\/ $D$
  from \eqref{eq:T-def}. Further, $B$ is simply called
  \emph{reversible} if it is\/ $\bs{p}\ts\ts$-reversible for some\/
  $\bs{p}>\bs{0}$.
\end{definition}

For a $\bs{p}\ts\ts$-reversible $B$, we see that $\bs{p}$ is a left
eigenvector with eigenvalue $\lambda$ if and only if
$\bs{1}^{\nts\trans}$ is a right eigenvector for $\lambda$. Also,
since $D^{1/2}$ is symmetric, $\bs{p}\ts\ts$-reversibility of $B$
implies
\[
  \bigl( D^{1/2} \nts B \ts D^{-1/2} \bigr)^{\trans}  = \,
  D^{-1/2} \nts  B^{\trans} \ts D^{1/2}  \, = \, D^{1/2} \nts B \ts D^{-1/2} ,
\]
as well as the analogous equation for positive powers of $B$, which 
shows the following important property.

\begin{fact}\label{fact:symm}
  Let\/ $B\in\Mat (d,\RR)$ be reversible for the probability vector\/
  $\bs{p}>0$. Then, $B$ is similar to a symmetric matrix, hence
  diagonalisable, and has real spectrum, $\sigma (B) \subset \RR$.
  Further, $B^n$ is reversible for any\/ $n\in\NN$. \qed
\end{fact}

We will be interested in reversible matrices $B$ that possess a real
matrix logarithm with particular properties, that is, in solutions of
$B=\ee^R$ with $R\in\Mat(d,\RR)$ and possibly further
restrictions. Let us first recall Culver's results \cite{Culver} on
the existence and uniqueness of real logarithms, formulated in terms
of the (complex) \emph{Jordan normal form} (JNF) of $B$.

\begin{fact}[Culver]\label{fact:Culver}
  A matrix\/ $B\in\Mat(d,\RR)$ has a real logarithm if and only if\/
  $B$ is non-singular and has the property that, in the JNF of\/ $B$,
  every elementary Jordan block with a negative eigenvalue occurs with
  even multiplicity.

  Further, $B$ has a unique real logarithm if and only if all
  eigenvalues of\/ $B$ are positive real numbers and no elementary
  Jordan block of the JNF of\/ $B$ occurs more than once. If\/ $B$ is
  diagonalisable, this is equivalent to\/ $B$ having simple, positive
  spectrum.  \qed
\end{fact}

When $B$ is diagonalisable with positive spectrum but repeated
eigenvalues, it has uncountably many real logarithms, as also
discussed in \cite{Culver}. Such a non-uniqueness is relevant for the
following reason.

\begin{lemma}\label{lem:dichotomy}
  Let\/ $B\in\Mat(d,\RR)$ satisfy\/ $B=\ee^Q = \ee^R$ with\/
  $Q,R \in \Mat(d,\RR)$, and consider\/
  $\varTheta \defeq \{ t \in \RR : \ee^{t \ts Q} = \ee^{t R}
  \}$. Then, either\/ $\varTheta = \RR$, which is equivalent with\/
  $Q=R$, or\/ $\varTheta \subset \RR$ is locally finite, which means
  that\/ $\varTheta \cap [a,a+r]$ is a finite set for every\/
  $a\in\RR$ and $r>0$.
\end{lemma}

\begin{proof}
  Obviously, $Q=R$ implies $\varTheta = \RR$, while the converse
  emerges from $B(t) = \ee^{t \ts Q}$ via $Q = \dot{B} (0)$, so the
  equivalence of $Q=R$ and $\varTheta = \RR$ is clear.
   
  Now, observe that $\varTheta$ is a closed subset of $\RR$, as
  follows from the continuity of the matrix exponential maps
  $t\mapsto \ee^{t \ts Q}$ and $t\mapsto \ee^{tR}$. Assume that
  $\varTheta$ contains a sequence $(t_m)^{}_{m\in\NN}$ of distinct
  points that converge in $\RR$, so
  $\lim_{m\to\infty} t_m = t^{}_{0}$. Then, $t^{}_{0} \in \varTheta$
  as well, which is to say that we also have
  $\ee^{t^{}_{0} \ts Q} = \ee^{t^{}_{0} R}$.
   
  Now, we can compare the series expansions around $t^{}_{0}$, which
  result in the matrix relation
  $ Q = R + \mathcal{O} (t_m - t^{}_{0} ) $ for every $m\in\NN$,
  interpreted elementwise. As $t_m$ converges towards $t^{}_{0}$, this
  is only possible if $Q=R$. So, when $Q\ne R$, the set $\varTheta$
  cannot contain any limit point, and must thus be locally finite,
  hence discrete and closed.
\end{proof}

Let us return to Culver's result.  Indeed, matrices with negative
eigenvalues can have a real logarithm, and this situation will still
occur in the realm of reversible matrices; see Fact~\ref{fact:roots} 
below.  However, if any of the
logarithms is to be reversible, one hits the following constraint from
the \emph{spectral mapping theorem} (SMT); compare
\cite[Thm.~9.4.6]{Lan}.

\begin{fact}\label{fact:rev-exp}
  If\/ $R\in\Mat(d,\RR)$ is reversible for\/ $\bs{p}>\bs{0}$, its
  exponential\/ $\ee^R$ is\/ $\bs{p}\ts\ts$-reversible as well, hence
  diagonalisable, and has positive spectrum,
  $\sigma (\ee^R) \subset \RR_{+}$.
\end{fact}

\begin{proof}
  If $R$ is $\bs{p}\ts\ts$-reversible, it is diagonalisable with
  $\sigma(R)\subset \RR$ by Fact~\ref{fact:symm}. All powers of $R$
  and their linear combinations are $\bs{p}\ts\ts$-reversible as well,
  hence also $\ee^R$, the latter by a standard continuity argument. 
  Here, $\ee^R$ is diagonalisable (by the same matrix as used for $R$), 
  and $\sigma(\ee^R) = \{ \ee^{\lambda} : \lambda \in \sigma(R)\} 
  \subset \RR_{+}$ is clear.
\end{proof}

Let us pause for a comment on the set of $\bs{p}\ts\ts$-reversible
matrices, defined as
\begin{equation}\label{eq:def-Ap}
    \cA_{\ts\bs{p}} \, \defeq \; \{ A \in \Mat (d,\RR) :
    p_i A_{ij} = p_j A_{ji} \text{ holds for all } 
    1\leqslant i,j \leqslant d \ts \}  \ts ,
\end{equation}
which contains $\one$ and is (topologically) closed. 
If $A,B \in \cA_{\ts\bs{p}}$, it is 
clear that also any real linear combination
of $A$ and $B$ lies in $\cA_{\ts\bs{p}}$. Moreover, for any 
$1\leqslant i,j \leqslant d$, we get
\begin{equation}\label{eq:Jordan}
   p^{}_i (AB)^{}_{ij} \, =  \sum_{k=1}^{d} p^{}_i A^{}_{ik} B^{}_{kj}
   \, =  \sum_{k=1}^{d} p^{}_k A^{}_{ki} B^{}_{kj} \, = 
   \sum_{k=1}^{d} p^{}_{j} B^{}_{jk} A^{}_{ki} \, = \, 
   p^{}_{j} (B\nts A)^{}_{ji} \ts ,
\end{equation}
which interchanges the matrix order. It then follows that
$AB + B\nts A$ is $\bs{p}\ts\ts$-reversible again, so that
$\cA_{\ts\bs{p}}$ is a real Jordan algebra, in line with the results 
from \cite{CS}. Note that we did not need $\bs{p}$ to be strictly positive 
in \eqref{eq:Jordan}, so $\cA_{\ts\bs{p}}$ can also be considered
for any $\bs{p}\geqslant \bs{0}$ via the detailed balance 
equations. Put together, we thus have the following result.

\begin{fact}\label{fact:p-alg}
  For any fixed probability vector\/
  $\bs{p} = (p^{}_{1} , \ldots , p^{}_{d} )$, the set\/
  $\cA_{\ts\bs{p}}$ from \eqref{eq:def-Ap}, with the product\/
  $A\circ\nts B \defeq \frac{1}{2} (AB + B\nts A)$, is a unital Jordan
  algebra.  \qed
\end{fact}

In fact, $\cA_{\ts\bs{p}}$ being closed under taking squares is 
equivalent with it being a Jordan algebra, because 
$A^2 =  A \circ \nts A $ and $A \circ \nts B  = \frac{1}{2} \bigl(
(A+ \nts B)^2 - A^2 - B^2\bigr)$. In either case, one then also
gets closure under taking any positive powers, as one sees
from $A^{n+1} = A \circ \nts A^n$ inductively, again for any
probability vector $\bs{p}$. We shall return to this in the
context of Markov matrices.

Now, if we are interested in reversible matrices with a reversible
real logarithm, Fact~\ref{fact:rev-exp} implies that we need to
restrict to matrices with positive spectrum. If $B\in\Mat(d,\RR)$ has
simple, positive spectrum, Fact~\ref{fact:Culver} asserts that there
is precisely one real logarithm, and this is the principal one. Since
such a $B$ is diagonalisable, we have
$B = U \diag (\lambda^{}_{1}, \ldots , \lambda^{}_{d}) \ts U^{-1}$ for
some invertible real matrix $U$. Then, we simply get the principal
logarithm as
\[
    \log (B) \, = \, U \diag \bigl( \log (\lambda_1) , \ldots,
        \log (\lambda_d) \bigr) \ts U^{-1} .
\]
When the spectrum possesses non-trivial multiplicities, there are 
several real logarithms, but we still have the following result.

\begin{lemma}\label{lem:principal}
  Let\/ $B \in \Mat(d,\RR)$ be\/ $\bs{p}\ts\ts$-reversible for\/
  $\bs{p}>\bs{0}$. If\/ $B$ has positive spectrum, its principal
  logarithm is a real matrix that is\/ $\bs{p}\ts\ts$-reversible as
  well.
\end{lemma}
    
\begin{proof}
  The principal logarithm of a real matrix with positive spectrum
  exists and is again real \cite[Thm.~1.31]{Higham}. By
  \cite[Thm.~11.1]{Higham}, it is given by
\[
  \log (B) \, = \int_{0}^{1} (B-\one) \bigl( t (B-\one) + \one
  \bigr)^{-1} \dd t \ts ,
\]
where the two matrices under the integral commute.

Now, we have $B^{\trans} = D \nts B D^{-1}$ from
Definition~\ref{def:reversible}, which implies
\[
\begin{split}
  \log (B)^{\trans} \ts & = \int_{0}^{1} \Bigl( \bigl( t (B-\one) +
  \one \bigr)^{\! \trans} \Bigr)^{-1} (B-\one)^{\trans} \dd t \, =
  \int_{0}^{1} D \ts \bigl( t (B-\one)
  + \one \bigr)^{-1} (B-\one) \ts D^{-1} \dd t \\
  & = \, D \!  \int_{0}^{1} (B-\one) \bigl( t (B-\one)+\one
  \bigr)^{-1}\dd t \, D^{-1} \, = \, D \log (B) \ts D^{-1}
\end{split}
\]
as claimed, where the penultimate step uses the commutativity
mentioned above.
\end{proof}

The occurrence of multiple solutions is related to the structure of
the \emph{commutant} of the matrix $B$, which is the matrix ring
\begin{equation}\label{eq:def-comm}
   \comm (B) \, \defeq \, \{ S \in \Mat(d,\RR) : [B,S] = \nix \} \ts .
\end{equation}
When the characteristic polynomial of $B$ is also its minimal one,
$\comm(B)$ is Abelian and given by the polynomial ring $\RR[B]$.
Otherwise, it contains further elements, which complicate matters; see
\cite[Sec.~12.4]{Lan} for details. Below, we shall see that some
natural conditions will ensure that at most one real logarithm of a
given $B$ is reversible.

\section{Reversible Markov matrices}\label{sec:main}

Let us now consider the case of Markov matrices, where we begin with a
mild extension of the reversibility notion. Note that we deviate from
the standard approach in \cite[Sec.~1.6]{LP} because we do not want to
restrict to irreducible Markov matrices.

\begin{definition}\label{def:weak-rev}
  A Markov matrix\/ $M$ is \emph{reversible} if it is reversible in
  the sense of Definition~$\ref{def:reversible}$. Further, it is
  called \emph{weakly reversible} if the detailed balance equations\/
  $p^{}_{i} M^{}_{ij} = p^{}_{j} M^{}_{ji}$ hold for all\/
  $1\leqslant i,j \leqslant d$ and some probability vector\/ 
  $\bs{p} \geqslant \bs{0}$, thus admitting zero entries.
  The corresponding notions are also used for Markov generators.
\end{definition}

\subsection{General structure}
If $M\in\cM_d$ is weakly reversible, $\bs{p}$ must be an equilibrium
vector of $M$.  Weak reversibility of $M\in\cM_d$ is equivalent to the
identity $D M = M^{\trans} D$ with the matrix $D$ from
Eq.~\eqref{eq:T-def}. Note that $D$ is not invertible if $M$ is only 
weakly reversible. In relation to $\cA_{\ts\bs{p}}$ from
Eq.~\eqref{eq:def-Ap}, for any fixed probability vector $\bs{p}$, 
one  can consider the set of (weakly) $\bs{p}\ts\ts$-reversible 
Markov matrices
\[
    \cM_{d,\bs{p}} \, \defeq \, \cM_d \cap \cA_{\ts\bs{p}} \ts ,
\]
which is convex and topologically closed. It is also closed under
taking squares, under taking arbitrary positive powers, and under
the Jordan product $M \ts {\circ} \ts M' = \frac{1}{2} ( M M' \nts + M' M )$.
In fact, these three properties are equivalent, because (weak)
reversibility is defined by a linear relation, so our earlier
argument used after Fact~\ref{fact:p-alg} still applies.

When $M$ is irreducible, weak
reversibility implies reversibility, and reversibility can be
established via Kolmogorov's loop criterion; see
\cite[Sec.~1.5]{Kelly} for details. At least for small state spaces,
this is also effective, while the computational effort quickly grows
with $d$; see \cite{BCHJ} for details and alternatives.

\begin{example}\label{ex:d-two}
  Let us consider $\cM_2$. Clearly, $\one\in\cM_2$ is (weakly)
  reversible for every probability vector $\bs{p}$. Otherwise,
  $\one \ne M \in \cM_2$ reads
  $M = \left( \begin{smallmatrix} 1-a & a \\ b & 1-b
    \end{smallmatrix} \right)$ with $a,b \in [0,1]$, not both zero.
  Then, the unique equilibrium vector is
  $\bs{p} = \frac{1}{a+b} (b,a)$, and $M$ is $\bs{p}\ts\ts$-reversible
  if $ab>0$ and weakly reversible otherwise.

  By Kendall's theorem \cite{King}, $M$ is embeddable if and only if
  $\det(M) >0$, equivalently $a+b<1$, and the embedding is
  unique. Indeed, one then has $M = \ee^Q$ with the principal matrix
  logarithm $Q = - \frac{\log(1-a-b)}{a+b} (M\nts - \one)$, which is a
  Markov generator when $a+b<1$. Moreover, $Q$ is (weakly) reversible
  if and only if $M$ is (weakly) reversible.

  On the other hand, there are reversible matrices that are not
  embeddable at all, such as
  $M = \left(\begin{smallmatrix} 0 & 1 \\ 1 & 0 \end{smallmatrix}
  \right)$, and we should expect more complicated cases in $\cM_d$
  with $d>2$.  \exend
\end{example}

\begin{example}\label{ex:equal-input}
  Let $\bs{x} \in\RR^d$ be a non-negative row vector and let
  $C_{\bs{x}}\in\Mat(d,\RR)$ denote the matrix with $d$ equal rows,
  each being $\bs{x}$. Then, $Q_{\bs{x}} = -x \one + C_{\bs{x}}$ with
  $x = \tr (C_{\bs{x}}) = x^{}_{1} + \ldots + x^{}_{d}$ is a Markov
  generator and $M_{\bs{x}} = (1-x) \one + C_{\bs{x}}$ a Markov
  matrix, the latter under the condition $x \leqslant \frac{d}{d-1}$;
  see \cite{BS2} for details. Such matrices are said to be of
  \emph{equal-input} type (or to be \emph{parent independent}).  
  Since $\bs{x}=\bs{0}$ gives
  $M_{\bs{0}}=\one$, which is (weakly) reversible for every
  probability vector $\bs{p}$, let us assume that $\bs{x}$ has at
  least one positive entry.

  Now, $M_{\bs{x}}$ with non-negative $\bs{x}$ and
  $0<x\leqslant \frac{d}{d-1}$ has a unique equilibrium vector, 
  $\bs{p} = \frac{\bs{x}}{x}$, and $M_{\bs{x}}$ is weakly reversible
  for this $\bs{p}$, as we always have
  $D \ts C^{}_{\bs{x}} = C^{\trans}_{\bs{x}} \ts D$ with $D$
  as in \eqref{eq:T-def}, and reversible
  when $\bs{x}> \bs{0}$. So, all equal-input Markov matrices are
  (weakly) reversible.  The embedding problem for this matrix
  class was studied in detail in \cite{BS2}. \exend
\end{example}

\begin{example}\label{ex:bad}
  Let $M \in \cM_d$ with $d\geqslant 2$ be tridiagonal, 
  with $a_i = \PP( i \rightarrow i{+}1)>0$ for $1\leqslant i < d$
  on the superdiagonal and  $b_i = \PP ( i \rightarrow i{-}1)>0$ 
  for $1 < i \leqslant d$ on the subdiagonal. This is then 
  the irreducible transition matrix of a simple birth and death process
  (or random walk)  in discrete time with states $\{ 1, 2, \ldots , d \ts \}$.
  Here, $M$ is always reversible, because the equations for its equilibrium
  vector directly equal the detailed balance equations. 
  
  For $d=2$ and $a^{}_{1} + b^{}_{2} < 1$, we get embeddability from
  Kendall's theorem, compare \cite{King,BS1}, but no other case is embeddable
  because it violates the transitivity property that $M_{ij} > 0$ and
  $M_{jk}>0$ also forces $M_{ik}>0$; see \cite[Prop.~2.1]{BS1} for details.
  To circumvent this well-known obstruction, one can consider, as is often 
  done, birth and death processes in continuous time, then using tridiagonal 
  rate matrices to describe them.
\exend
\end{example}

Recall that $M\in\cM_d$ is reversible if it is $\bs{p}\ts\ts$-reversible
for some $\bs{p}>\bs{0}$. Such a $\bs{p}$ must be an equilibrium
vector.  The existence (and uniqueness) of such a $\bs{p}$ is clear 
when $M$ happens to be irreducible, by the Perron--Frobenius 
theorem \cite[Thm.~2.1.4]{BP}.  Since we do not restrict to 
irreducible $M$, we also need the following result, which we 
could not trace in the literature.

\begin{fact}\label{fact:pos-p}
  A matrix\/ $M\in\cM_d$ has a strictly positive equilibrium vector if
  and only if\/ $M$ is the direct sum of irreducible Markov matrices,
  hence a diagonal block matrix with irreducible blocks, possibly
  after a suitable permutation of the states. The equilibrium vector
  is unique if and only if\/ $M$ is irreducible itself.
\end{fact}

\begin{proof}
  Any $M\in\cM_d$ has spectral radius $1$ and satisfies
  $M \ts \bs{1}^{\nts \trans} = \bs{1}^{\nts \trans}$, which is
  strictly positive. If $M$ is irreducible, also the corresponding
  left eigenvector $\bs{p}$ is strictly positive, by the
  Perron--Frobenius theorem. When normalised such that
  $\bs{p} \cdot \bs{1}^{\nts \trans} = 1$, this is the equilibrium
  vector. If $M = \oplus_{i=1}^{s} M^{(i)}$ with $s>1$ and irreducible
  $M^{(i)}$ of dimension $d^{(i)}$, every block has its unique
  equilibrium vector $\bs{p}^{(i)}$ of length $d^{(i)}$. Then, each
  $\bs{0} \oplus \dots \oplus \bs{0} \oplus \bs{p}^{(i)} \oplus \bs{0}
  \oplus \dots \oplus \bs{0}$ is an equilibrium vector of $M$, and any
  convex combination of these $s$ probability vectors is again an
  equilibrium vector, and among them are many strictly positive ones.

  It remains to see why this is the only possibility of $M\in\cM_d$
  with $\bs{p} M = \bs{p}>\bs{0}$. Thus, we need a criterion for the
  existence of a strictly positive left eigenvector. Since a strictly
  positive right eigenvector exists, we can invoke
  \cite[Thm.~2.3.14]{BP}, which implies that $\lambda = 1$ has
  strictly positive left \emph{and} right eigenvectors if and only if, after
  some permutation, $M = \bigoplus_{i=1}^{s} M^{(i)}$ with all
  $M^{(i)}$ irreducible. This is equivalent to all communication
  classes being both basic and final, which brings the Frobenius
  normal form to this diagonal block form.
\end{proof}

A reversible Markov matrix, up to permutations of the states, must
then be such a block-diagonal matrix with irreducible blocks, and this
is the reason why most treatments of reversibility restrict to one
block, and hence to irreducible matrices. In the context of the
embedding problem and the relation to weakly reversible extensions,
this is neither necessary nor advantageous.

\begin{remark}\label{rem:freedom}
  When $M\in\cM_d$ is irreducible and reversible, it must be
  reversible for the unique equilibrium vector $\bs{p}$ of $M$, which
  is strictly positive. When $M$ is reducible and reversible for some
  $\bs{p}>\bs{0}$, this is not the only equilibrium vector. Since
  $M = \bigoplus_{i=1}^{s} M^{(i)}$ as in the proof of
  Fact~\ref{fact:pos-p}, each irreducible block $M^{(i)}$ must be
  reversible on its own. Consequently, $M$ must be (weakly) reversible
  for \emph{each} of its equilibrium vectors, which gives some freedom
  to choose a suitable one when comparing with nearby matrices.
  
  Note that the converse need not hold: A weakly reversible matrix $M$ 
  can be a direct sum of a reversible and a non-reversible one, where 
  both can still be irreducible. Then, there is also a strictly positive
  equilibrium vector, but $M$ can neither be reversible with respect to
  this one, nor with respect to any other strictly positive one.
  \exend
\end{remark}

At this point, we have gained access to the following classes of matrices.

\begin{example}\label{ex:symm}
  Consider the convex set of \emph{doubly stochastic} matrices within
  $\cM_{d}$, which are the Markov matrices with both unit row sums 
  and unit column sums. By the Birkhoff--von Neumann theorem, 
  they form a convex polytope with the $d!$ permutation matrices as its 
  vertices (or extremal points), so each doubly stochastic matrix is a
  convex combination of (some of) these.
  
  If $M$ is any doubly stochastic matrix, $\bs{1}$ is a left eigenvector 
  for the leading eigenvalue $\lambda = 1$, so $\bs{p} = \frac{1}{d} \bs{1}$
  is an equilibrium vector, which gives $D = \one$. All \emph{symmetric} Markov 
  matrices are $\bs{p}\ts\ts$-reversible, with $\widetilde{M} = D^{-1} M^{\trans} D
  = M^{\trans} = M$. Moreover, by Remark~\ref{rem:freedom}, they are 
  the only reversible ones among the doubly stochastic
  matrices, though there can be further weakly reversible
  matrices that are not symmetric.   \exend
\end{example}  
  
\subsection{Embedding}
Let us now take a closer look at embeddability, where we start with
a precise version of the notion alluded to in the introduction.

\begin{definition}\label{def:embeddable}
  A (weakly) reversible Markov matrix\/ $M$ is called \emph{embeddable}
  if it satisfies\/ $M = \ee^Q$ for some Markov generator\/ $Q$.
  Further, $M$ is \emph{reversibly embeddable} if\/ $M = \ee^Q$
  holds for a Markov generator\/ $Q$ that is itself (weakly) reversible.
\end{definition}
  
For embeddability, we thus need a real logarithm that is also a
rate matrix. To obtain a practically useful criterion, let us begin with 
the following observation.

\begin{fact}\label{fact:principal}
  Let\/ $M \in \cM_d$ be a Markov matrix with positive spectrum,
  $\sigma(M)\subset\RR_{+}$, and let\/ $m\leqslant d$ be the degree of
  its minimal polynomial. Then, $M$ has at least one real logarithm,
  namely its principal matrix logarithm,
\[
    \log(\one+A)  \, = 
    \sum_{n=1}^{\infty} \myfrac{(-1)^{n-1}}{n} \ts A^n ,
\]
where\/ $A = M \nts - \one$ is a rate matrix with spectral radius\/
$\varrho^{}_{A} < 1$. Here, $\log(\one+A)$ is a matrix with zero row
sums that is a polynomial in $A$ of degree\/ $m-1$ with real
coefficients.

Further, when\/ $\sigma(M) \subset \RR_{+}$ is simple, $\log(\one+A)$
is the only real matrix logarithm of\/ $M$.
\end{fact}

\begin{proof}
  Under the assumptions, we have $\sigma(M) \subset (0,1]$, hence
  $\sigma(A) \subset (-1,0]$, and $\varrho^{}_{A}<1$ is clear. This
  implies the convergence of the series. Since $M$ has unit row sums,
  $A$ has zero row sums, as do all positive powers of $A$. As this
  property is preserved in the limit by continuity, $\log(\one+A)$ has 
  zero row sums as well.

  Since the minimal polynomials of $M$ and $A$ have the same degree,
  the claim on the polynomial form is a consequence of the
  Cayley--Hamilton theorem; compare \cite[Thm.~3.3.1]{HJ}.

  The final uniqueness statement is a consequence of Culver's theorem
  (Fact~\ref{fact:Culver}).
\end{proof}

It is clear that a reversible Markov matrix has real spectrum, because
it is similar to a symmetric matrix. Such a matrix can still have
negative eigenvalues. If all negative eigenvalues have even
multiplicity, a real matrix logarithm exists and embeddability may
occur, but not with a reversible generator, as the following result
shows.

\begin{lemma}\label{lem:url}
  Let\/ $M\in\cM_d$ and let\/ $\bs{p}>\bs{0}$ be a probability
  vector. Assume that a\/ $\bs{p}\ts$-reversible real matrix
  logarithm\/ $R$ of\/ $M$ exists, so\/ $M = \ee^R$. Then, $M$ is\/
  $\bs{p}\ts\ts$-reversible with $\sigma (M) \subset
  \RR_{+}$. Further, no other\/ $\bs{p}\ts\ts$-reversible real
  logarithm of\/ $M$ exists, and\/ $\ts R = \log(\one+A)$ is the
  principal matrix logarithm from Fact~$\ts\ref{fact:principal}$.
\end{lemma}

\begin{proof}
  When $M = \ee^R$ with a $\bs{p}\ts\ts$-reversible $R$,
  Fact~\ref{fact:rev-exp} implies the first claim.
  
  It is clear that the rate matrix $A = M - \one$ is
  $\bs{p}\ts\ts$-reversible, as are the powers $A^n$ for $n\in\NN$.
  Consequently, $L=\log(\one+A)$ from Fact~\ref{fact:principal} is a
  real logarithm of $M$ with zero row sums that is also
  $\bs{p}\ts\ts$-reversible, in line with Lemma~\ref{lem:principal}.
  
  In the generic case that $M$ has simple spectrum, there is precisely
  one real logarithm of $M$ by Culver's theorem
  (Fact~\ref{fact:Culver}), which must be the existing principal one,
  so we get $R=L=\log(\one+A)$.

  In general, we know that $M$ and $L$ commute, and both are
  $\bs{p}\ts\ts$-reversible matrices that are simultaneously
  diagonalisable, with positive and with real spectrum,
  respectively. Moreover, their diagonal forms match, in the sense
  that $\lambda\in\sigma(M)$ and $\log(\lambda)\in\sigma(L)$ 
  share the same (algebraic and geometric) multiplicity. 
  This implies that they have the same
  commutant, $\comm(M) = \comm(L)$, as defined in 
  \eqref{eq:def-comm}; see \cite{Lan} for background and details.

  Now, let $Q$ be any real solution of $M=\ee^Q$. Then, $Q$ must
  commute with $M$ and hence also with $L$. As $\ee^Q$ is invertible,
  we get
\[  
    \one \, = \, \ee^L \ee^{-Q} \, = \, \ee^{L-Q},
\]
and $L-Q$ is a real logarithm of $\one$. If $Q$ is also
$\bs{p}\ts\ts$-reversible, then so is $L-Q$, which thus is
diagonalisable with real spectrum. By the SMT, as $0$ is the only real
logarithm of $1$, the spectrum of the diagonalisable matrix $L-Q$ must
consist of $0\ts\ts$s entirely, which implies $L-Q=\nix$. We thus get
$Q=L$, which also applies to the matrix $R$ from the assumption.
\end{proof}

\begin{remark}\label{rem:sets-of-p}
  Let us note that $M = \ee^Q$ with $Q$ a Markov generator
  implies that $M$ and $Q$ share the same set of equilibrium vectors.
  Clearly, $\bs{p} \ts Q = \bs{0}$ implies $\bs{p} M = \bs{p}$. The converse
  direction depends on the fact that any eigenvalue of a Markov generator
  is either $0$ or has a strictly negative real part, see
  \cite[Prop.~2.3]{BS1}, thus can never be purely imaginary, hence
  not of the form $k \ts 2 \pi \ii$ for $k\ne 0$. 
  Since the algebraic and the geometric multiplicity of $1\in \sigma (M)$
  is the same \cite[Fact~2.2]{BS3}, the spectrum of $Q$ must contain
  $0$ with the identical (algebraic and geometric) multiplicity by the SMT. 
  Consequently, any (real) eigenvector of $M$ with $\bs{x} M = \bs{x}$ 
  must also satisfy $\bs{x} Q = \bs{0}$, which includes all equilibrium vectors.
  In particular, in this situation, $M$ has a strictly positive one if and only
  if $Q$ does.
  
  Let us note that the above argument critically depends on $Q$ being
  a rate matrix. Indeed, it can already fail if $Q$ has only zero row sums,
  but violates non-negativity of its off-diagonal elements (such as 
  $Q = \frac{2 \pi}{\sqrt{3}}(Q_{+} \nts - Q_{-})$ with $Q^{}_{\pm}$ 
  from Example~\ref{ex:strange} below, which has spectrum $\sigma (Q) = 
  \{ 0, \pm \ii \}$). Interestingly,
  Fact~\ref{fact:pos-p} now also has the consequence that a 
  reversible and embeddable Markov matrix, possibly after a suitable 
  permutations of the states, must be the direct sum of totally positive 
  (hence in particular primitive) block matrices, because the exponential of an
  irreducible Markov generator is a totally positive matrix.    \exend
\end{remark}

Let us continue with a straightforward consequence of
Lemma~\ref{lem:url}.

\begin{prop}\label{prop:rev-simple}
  If\/ $M\in\cM_d$ is a reversible Markov matrix with simple, positive
  spectrum, the following statements are equivalent, where\/
  $A = M \nts - \one$ as before.
\begin{enumerate}\itemsep=2pt
  \item $M$ is embeddable.
  \item $M$ is reversibly embeddable.
  \item $Q=\log(\one+A)$ is a Markov generator.
\end{enumerate}  
\end{prop}

\begin{proof}
  Under the assumption, there is precisely one real matrix logarithm,
  namely the principal one from Fact~\ref{fact:principal}, which is
  $Q$ as stated. This $Q$ is reversible with the same $\bs{p}$ as $M$.

  Consequently, $M$ is embeddable if and only if this $Q$ is a Markov
  generator.  If so, it is automatically a reversible generator, and
  we are done.
\end{proof}

When the spectrum fails to be simple, further real logarithms exist, and
we shall see explicit cases below in Examples~\ref{ex:strange} and 
\ref{ex:super-strange}. However, another consequence of Lemma~\ref{lem:url} 
in conjunction with Remark~\ref{rem:freedom} is the following.

\begin{coro}\label{coro:rev-embed}
  A reversible Markov matrix\/ $M = \one + A \in \cM_d$ with positive
  spectrum is reversibly embeddable if and only if\/
  $Q = \log(\one+A)$ is a Markov generator.
  \qed
\end{coro}

When applied to the class of symmetric matrices, compare
Example~\ref{ex:symm}, this gives the
following generalisation of \cite[Thm.~6.5]{BS1}.

\begin{coro}\label{coro:symm}
  A symmetric\/ $M = \one + A\in \cM_d$ is symmetrically embeddable 
  if and only if\/ $\sigma (M) \subset \RR_{+}$ and the symmetric
  matrix\/ $Q = \log(\one+A)$ is a Markov generator.  \qed
\end{coro}

\subsection{Concrete conditions}
Let us next derive a simple criterion for reversible embeddability, that
is, for $\log(\one+A)$ to be a Markov generator. Let $M\in\cM_d$ be
reversible with positive spectrum, possibly with multiplicities, and
assume that the degree of its minimal polynomial is $m$, where
$m\leqslant d$. Then, by Fact~\ref{fact:principal}, there are real numbers
$\alpha^{}_{1}, \ldots , \alpha^{}_{m-1}$ such that
\begin{equation}\label{eq:L-poly}
   L \, = \, \log(\one+A) \, = \, \sum_{k=1}^{m-1} \alpha^{}_{k} A^{k},
\end{equation}
where no term with $A^0 = \one$ occurs because $L$ has zero row sums.
Since $\sigma (M) \subset (0,1]$, we can write the eigenvalues of $M$
in descending order as
$\lambda^{}_{0} = 1 > \lambda^{}_{1} > \cdots > \lambda^{}_{m-1} > 0$
and those of $A$ as
$\mu^{}_{0} = 0 > \mu^{}_{1} > \cdots > \mu^{}_{m-1}$, where
$\mu^{}_{i} = \lambda^{}_{i}-1$. The corresponding eigenvalues of $L$
are the unique real logarithms
$0 = \log(\lambda^{}_{0}) > \log (\lambda^{}_{1}) > \dots > \log
(\lambda^{}_{m-1})$.  The SMT now gives $m-1$ conditions for the
eigenvalues. They can be written in matrix form as
\begin{equation}\label{eq:VDM}
  \begin{pmatrix}
    \mu^{}_{1} & \mu^{2}_{1} & \cdots & \mu^{m-1}_{1} \\
    \mu^{}_{2} & \mu^{2}_{2} & \cdots & \mu^{m-1}_{2} \\
    \vdots &  & & \vdots \\
    \mu^{}_{m-1} & \mu^{2}_{m-1} & \cdots & \mu^{m-1}_{m-1} 
  \end{pmatrix} \begin{pmatrix} \alpha^{}_{1} \\
      \alpha^{}_{2} \\ \vdots \\ \alpha^{}_{m-1}
  \end{pmatrix} \, = \, \begin{pmatrix} \log(\lambda^{}_{1}) \\
    \log(\lambda^{}_{2}) \\ \vdots \\ \log(\lambda^{}_{m-1}) 
  \end{pmatrix} .
\end{equation}
We note that the number of equations is one lower than in the 
approach of \cite{Chen}.  The matrix on the left has determinant
$\bigl( \prod_i \mu^{}_{i} \bigr) \prod_{k>\ell} (\mu^{}_{k} -
\mu^{}_{\ell})$ and is non-singular because the $\mu^{}_{i}$ with
$1\leqslant i \leqslant m-1$ are non-zero and distinct. So, the
coefficients $\alpha^{}_{i}$ are uniquely determined and real; see
\cite[Sec.~5]{BS2} for details and an explicit formula for the inverse
of this variant of the Vandermonde matrix. We thus have established
the following result.

\begin{prop}\label{prop:rev-embed}
  A reversible Markov matrix\/ $M\in\cM_d$ with a minimal polynomial
  of degree\/ $m\leqslant d$ is reversibly embeddable if and only if
  the following two conditions are satisfied.
\begin{enumerate}\itemsep=2pt
  \item The spectrum of\/ $M$ is positive, $\sigma(M) \subset \RR_{+}$.
  \item The coefficients\/ $\alpha^{}_{1}, \ldots , \alpha^{}_{m-1}$
    that solve Eq.~\eqref{eq:VDM} are such that\/ $L = \log(\one+A)$
    from Eq.~\eqref{eq:L-poly} is a Markov generator.
\end{enumerate}
In this case, $L$ is unique in the sense that no other reversible real 
logarithm of\/ $M$ exists.  \qed
\end{prop}

By \cite[Prop.~7]{King}, for any $2\leqslant d\in\NN$, the set $\cE_d$
of embeddable Markov matrices is relatively closed within the set
$\{ M\in \cM_d : \det (M) > 0 \}$. In contrast, the set of reversible
Markov matrices is \emph{not} relatively closed, because the limit of
a converging sequence of reversible Markov matrices, provided that
also the corresponding probability vectors converge, might only be
weakly reversible. 

Likewise, we know from \cite[Thm.~7]{Davies} that
the set of embeddable Markov matrices with simple spectrum is
relatively open and dense in $\cE_d$. In contrast, this is no longer
true for embeddable reversible Markov matrices, because they have real
spectrum but can never have simple negative eigenvalues. Nevertheless,
we still get the following result.

\begin{theorem}\label{thm:pos-rev}
  If\/ $M\in\cM_d$ is reversible for some\/ $\bs{p}>\bs{0}$, the
  following are equivalent.
\begin{enumerate}\itemsep=2pt
\item $M$ is embeddable and has positive spectrum.
\item $M$ is embeddable with a\/ $\bs{p}\ts\ts$-reversible Markov
  generator.
\item $M$ has positive spectrum and its principal logarithm,
  $L = \log (\one + A)$ with\/ $A = M \nts - \one$, is a Markov
  generator.
\end{enumerate}
In this case, $L$ is the only reversible matrix logarithm of\/ $M$.

Beyond this case, when\/ $d \geqslant 3$, a reversible Markov matrix
can still be embeddable, but never with a reversible Markov generator.
\end{theorem}

\begin{proof}
  The implication $(2) \Rightarrow (1)$ is clear from
  Fact~\ref{fact:rev-exp}, while $(2) \Leftrightarrow (3)$ follows
  from Lemma~\ref{lem:url} by standard arguments.

  Next, we need to show that $(1) \Rightarrow (2)$. To this end,
  assume that $M\in\cM_d$ with $\sigma(M) \subset \RR_{+}$ is
  $\bs{p}\ts\ts$-reversible, and that $M = \ee^Q$ for some Markov
  generator $Q$. If $Q$ is $\bs{p}\ts\ts$-reversible itself, there is
  nothing to show, so assume that $Q$ is not reversible. By
  \cite[Thm.~7]{Davies}, we know that there are embeddable Markov
  matrices with simple spectrum in any neighbourhood of $M$ within
  $\cM_d$. In fact, there are also reversible ones among them, as
  follows from the following deformation argument, which is taken from
  the proof of \cite[Thm.~3.1]{BCHJ}.

  Let us first assume that $M$ is irreducible. Then, if $M$ is
  reversible, we can modify any of its rows by multiplying it with an
  arbitrary non-negative number such that the sum of the non-diagonal
  elements remains $\leqslant 1$, followed by modifying the diagonal
  element such that the row sum is again $1$. Similarly, we can 
  multiply any column with a number in $(0,1)$, so no row sum
  can exceed $1$, and then re-adjust all diagonal elements for unit 
  row sums.
  
  As follows from Kolmogorov's
  loop criterion, compare \cite{Kelly, BCHJ}, none of these operations
  destroys reversibility, as long as the resulting equilibrium vector
  (which may change) remains positive, which is certainly the case for
  sufficiently small deformations. This gives enough freedom to find,
  in any neighbourhood of $M$, a reversible matrix $M'$ with simple
  spectrum that is also embeddable. Then, the corresponding Markov
  generator is unique and given by the principal matrix logarithm of
  $M'$,

  More generally, $M$ reversible implies via Fact~\ref{fact:pos-p}
  that $M$ is the direct sum of irreducible Markov matrices. By
  permuting the states appropriately, we may assume that $M$ has
  diagonal block form. We can now employ the above deformation
  argument to each block separately. Since we have some freedom to
  choose $\bs{p}>\bs{0}$, as explained in Remark~\ref{rem:freedom}, we
  obtain the corresponding conclusion: in any neighbourhood of
  $M$, we can find an embeddable, reversible matrix $M'$ with simple
  spectrum and a nearby probability vector. For the latter,  without loss 
  of generality, we may assume that none of its entries is smaller than
  $\epsilon \defeq \frac{1}{2} \min_{i} p^{}_i > 0$, say.  The Markov 
  generator for $M'$ is again the principal matrix logarithm, which 
  has the same equilibrium vectors as $M'$; compare 
  Remark~\ref{rem:sets-of-p}.
  
  Now, we can select a converging sequence of embeddable reversible
  Markov matrices with simple spectrum that converges to $M$ such that
  also the corresponding equilibrium vectors converge to some
  $\bs{p}> \bs{0}$, which is possible by a standard compactness 
  argument involving the constant $\epsilon>0$.  
  All of these matrices are embeddable via 
  their principal logarithm, by Proposition~\ref{prop:rev-simple}, which is
  another converging sequence, this time of Markov generators. The
  limit is still a Markov generator, and it is reversible. Then, it
  must be the principal logarithm $L$ of $M$ by
  Corollary~\ref{coro:rev-embed}, which proves the claim.
  
  The uniqueness of $L$ follows from
  Proposition~\ref{prop:rev-simple} in conjunction with
  Lemma~\ref{lem:url} and Remark~\ref{rem:freedom}, because 
  $L$ must be  $\bs{p}\ts\ts$-reversible and thus also reversible 
  for any other equilibrium vector of $M$.
 
  The last claim is made explicit in Example~\ref{ex:strange} 
  below, which establishes the statement in line with 
  Lemma~\ref{lem:url}.
\end{proof}

\subsection{Multiple embeddings}
With the above results, Proposition~\ref{prop:rev-embed} becomes the right
tool to test all reversible Markov matrices with positive spectrum for
embeddability. When the spectrum is degenerate, we do not get
uniqueness. Let us look at this case more closely now.  In general,
when a reversible Markov matrix has more than one real logarithm,
non-reversible ones must come in pairs. Indeed, let
$M = \widetilde{M} = D^{-1} M^{\trans} D$ for some probability
vector $\bs{p}>\bs{0}$ with $D=D_{\bs{p}}$ as above, where $\bs{p}$
then is an equilibrium vector for $M$. Now, if $M = \ee^R$ for some
$R \in\Mat(d,\RR)$ and $\widetilde{R} = D^{-1} R^{\trans} D$, we
get
\[
  M \, = \, \widetilde{M} \, = D^{-1} \bigl(\ee^R\bigr)^{\!\trans}
  D \ts = \, \exp \bigl( D^{-1} R^{\trans} D\bigr) \, = \, \ee^{\widetilde{R}} ,
\]
where $\widetilde{R}$ is a Markov generator if and only if $R$ is one.
Indeed, if $R$ is a Markov generator, it must satisfy $\bs{1} R^{\trans}
= \bs{0}$ and share the equilibrium vector $\bs{p}$ with $M$,
so $\bs{p} R = \bs{0}$, by Remark~\ref{rem:sets-of-p}. Then, 
$\widetilde{R}$ has non-negative off-diagonal entries because $R$
does and $D$ is a non-negative matrix, and one gets
$\bs{p} \widetilde{R} = \bs{1} R^{\trans} D = \bs{0} D = \bs{0}$. 
Likewise, $\bs{1} \widetilde{R}^{\ts\trans} = \bs{p} R \ts D^{-1} =
\bs{0}  D^{-1}  = \bs{0}$, so $\widetilde{R}$ is a
rate matrix. The converse direction follows from 
$R = D^{-1} \widetilde{R}^{\ts\trans} D$ in the same way.

So, let us assume that $R$ is a rate matrix. Then,
we either have $\widetilde{R}=R$, which is thus reversible, or we
get a $\bs{p}\ts\ts$-balanced pair with $\widetilde{R}\ne R$, where
each is the time reversal of the other.
Clearly, $R+\nts\widetilde{R}$ is $\bs{p}\ts\ts$-reversible, as
$\,\widetilde{.}\,$ is an involution.  Then, if
$[\widetilde{R}, R]=\nix$, we get
\[
   M^2 \, = \, \ee^R \ee^{\widetilde{R}} \ts = \, \ee^{R+\nts\widetilde{R}},
\]
and Fact~\ref{fact:rev-exp} implies that $M^2$ is a reversible Markov
matrix with positive spectrum. Now, $M^2$ has a unique square root
with positive eigenvalues by \cite[Cor.~6]{JOR}, say $M'$. If $M$ has
positive spectrum itself, we get $M=M'$; otherwise, $M$ and $M'$ are
different square roots of $M^2$. One consequence is the following.

\begin{lemma}\label{lem:pair}
  Let\/ $M\in\cM_d$ be reversible, and embeddable via a 
  $\bs{p}\ts\ts$-balanced pair of Markov generators, $R$ and\/ 
  $\widetilde{R}$, with\/ $[\widetilde{R}, R \ts ]=\nix$ and\/
  $R\ne\widetilde{R}$. Then,
  $R+ \nts \widetilde{R}$ is a $\bs{p}\ts\ts$-reversible Markov generator
  and $M^2 = \ee^{2R} = \ee^{2\widetilde{R}}$ is also reversibly embedded 
  as\/ $M^2 = \exp(R + \nts \widetilde{R} \,)$. Further, $M^2$ has positive
  spectrum and possesses a unique Markov square root with positive
  spectrum. This square root is reversibly embedded via the
  Markov generator\/ $\frac{1}{2} (R + \nts \widetilde{R} \, )$.  \qed
\end{lemma}

Note that this result also applies if we start from a Markov matrix with
negative eigenvalues, which can be embeddable when each negative
eigenvalue appears with even multiplicity. Clearly, $d=3$ is the
smallest dimension where this can occur, and the embedding 
for such matrices was only settled relatively recently in \cite{Carette}; see
\cite[Sec.~3.1]{BS3} for a simplified derivation.

\begin{example}\label{ex:strange}
  Set $\epsilon = \ee^{-\pi \sqrt{3}}\approx 0.163034$ and consider
  the symmetric Markov matrix
\[
  M \, = \, \myfrac{1}{3} \begin{pmatrix}
    1-2\epsilon & 1+\epsilon & 1+\epsilon \\
    1+\epsilon & 1-2\epsilon & 1+\epsilon \\
    1+\epsilon & 1+\epsilon & 1-2\epsilon \end{pmatrix},
\]
which is reversible with $\bs{p} = \frac{1}{3} (1,1,1)$, hence
$D_{\bs{p}} = \frac{1}{3} \one$. Here, $M$ cannot be reversibly 
embeddable, because its spectrum is $\sigma(M) = \{1,
-\epsilon,-\epsilon\}$, so it lies outside the case covered
by Corollary~\ref{coro:symm}. Still, it is embeddable as
$M = \ee^{Q_{+}} = \ee^{Q_{-}}$ with the two Markov generators
\[
  Q^{}_{+} \, = \, \myfrac{2 \pi}{\sqrt{3}\ts} \begin{pmatrix}
    -1 & 1 & 0 \\ 0 & -1 & 1 \\ 1 & 0 & -1 \end{pmatrix}
  \quad \text{and} \quad
   Q^{}_{-} \, = \, \myfrac{2 \pi}{\sqrt{3}\ts} \begin{pmatrix}
    -1 & 0 & 1 \\ 1 & -1 & 0 \\ 0 & 1 & -1 \end{pmatrix},
\]
see \cite{BS1,BS3} and references therein for details. An 
independent approach is given in Appendix~\ref{sec:app}. Clearly,
$Q^{}_{+} = Q^{\trans}_{-}$, so they form a $\bs{p}\ts\ts$-balanced
pair, and it is easy to check that they commute,
$[Q^{}_{+}, Q^{}_{-}]=\nix$, which means that Lemma~\ref{lem:pair}
applies.
  
Now, $Q^{}_{+} \nts + \ts Q^{}_{-}$ is a $\bs{p}\ts\ts$-reversible Markov
generator of equal-input type (see Example~\ref{ex:equal-input}), and
we get $M^2 = \exp (Q^{}_{+} \nts + Q^{}_{-})$ with spectrum
$\{ 1, \epsilon^2,\epsilon^2\}$. The unique square root of $M^2$ with
positive spectrum $\{ 1, \epsilon,\epsilon\}$ is the matrix $M'$ that
emerges from $M$ by changing $\epsilon$ to $-\epsilon$. Note that
$M^2 = \ee^{2 Q_{+}} = \ee^{2 Q_{-}}$ is also reversibly embedded via 
$Q^{}_{+} \nts + Q^{}_{-}$, so we have an example
with positive spectrum and three embeddings.

More generally, the equal-input matrices from \cite[Prop.~3.8]{BS3} 
have a double negative eigenvalue and are reversible. They are also
embeddable, at least with one $\bs{p}\ts\ts$-balanced pair
$Q^{}_{\pm}$ of Markov generators, though there may be several
pairs, as derived in \cite{BS3} and in Appendix~\ref{sec:app}.
One can check that they always
satisfy $[Q^{}_{+} , Q^{}_{-}] = \nix$, so we are in the same
situation as above. As was derived in \cite[Sec.~3.1]{BS3}, this is
the only situation with a double negative eigenvalue that can occur
for $d=3$. \exend
\end{example}

To continue with multiple negative eigenvalues, let us consider
$B = - \one \in \mathrm{SL} (2, \RR)$, which is reversible for
$\bs{p} = \bigl( \frac{1}{2}, \frac{1}{2} \bigr)$ via 
$\widetilde{B} = B^{\trans} = B$.  The matrix $B$ has multiple 
real logarithms by Fact~\ref{fact:Culver}, which we now determine 
completely. Clearly, if $B = \ee^R$, the spectrum of $R \in \Mat(2,\RR)$ 
must be $\sigma (R) = \{ \pm (2m+1) \pi \ii \}$ for some arbitrary, but
fixed $m\in\ZZ$. We may thus write $R = (2m+1) I$, where $I$ is a
real matrix with $\sigma (I) = \{ \pm \ii \}$ and $I^2 = - \one$. In 
particular,  $\det (I) = 1$ and $I \in \mathrm{SL} (2, \RR)$. Also, one gets
\[
    \exp \bigl( (2m+1) \pi I \bigr) \, = \,
    \cos \bigl( (2m+1) \pi \bigr) \one + 
    \sin \bigl( (2m+1) \pi \bigr) I \, = \, - \one \ts ,
\]
and we have reduced our problem to finding all real square roots
of $-\one$.

Set $ I = 
\left( \begin{smallmatrix} a & b \\ c & d \end{smallmatrix} \right)$
and consider $I^2 = -\one$. As one can easily verify, $b=c=0$ does 
not lead to any real solution, so we need $a+d=0$ together with
$bc = - (a^2 +1) < 0$, which gives
\begin{equation}\label{eq:roots}
    I \, = \, \begin{pmatrix} a & b \\ - \frac{a^2 + 1}{b} & -a \end{pmatrix} .
\end{equation}
Each such $I$ together with $I^{\trans}$ gives a $\bs{p}\ts\ts$-balanced
pair, with $I \ne I^{\trans}$. Moreover, one has $[I , I^{\trans}] = \nix$
if and only if $a=0$ together with $b = \pm 1$, in which case one
has $I^{\trans} = I^3$. Put together, we have the following.

\begin{fact}\label{fact:roots}
   The matrix\/ $B = -\one \in \mathrm{SL} (2,\RR)$ is\/
   $\bs{p}\ts\ts$-reversible for\/ $\bs{p} = \bigl( \frac{1}{2},
   \frac{1}{2} \bigr)$, but has no reversible real logarithm.
   Instead, its real logarithms come in\/ $\bs{p}\ts\ts$-balanced
   pairs\/ $\{ R, R^{\trans} \}$ with\/ $R = (2m+1) \pi I$ for 
   some\/ $m\in\ZZ$ and\/ $I$ as in \eqref{eq:roots},
   with\/ $a,b \in \RR$ and\/ $b\ne 0$. 
   
   They commute if and only if\/ $a=0$ and\/ $b = \pm 1$.  \qed
\end{fact}

Now, we need to see whether the analogous situation occurs
among Markov matrices. The non-commutative case cannot 
happen for $d=3$, as discussed in Example~\ref{ex:strange}.
Fortunately, it suffices to search within the class of doubly
stochastic matrices from Example~\ref{ex:symm}, which allows
an approach based on permutation symmetries. Since we need 
some cyclic structure in the rate matrix (to obtain complex
eigenvalues), but also non-commutativity, we replace the cyclic 
group $C_3$, which underlies our previous example as detailed in 
Appendix~\ref{sec:app}, by the dihedral group 
\[
     D_4 \, = \,  \langle a, b : a^4 = b^2 = (ab)^2 = e \rangle \ts , 
\]
where $e$ denotes the neutral element. Concretely, we use
the permutations $a=(1234)$ and $b=(12)(34)$. Then, starting 
from the standard 4D permutation representation, we consider 
the algebra generated by these matrices, and make an ansatz for 
our generator that is mildly quadratic in the generating elements, 
namely $Q=\alpha (a-e) + \beta (b-e) + \gamma (a^2-e)$, and 
determine the parameters such that the spectrum becomes
$\sigma (Q) = \{ 0, -2 \ts \sqrt{3} \pi, \pm \ts \pi \ii - 2 \sqrt{3} \ts \pi \}$.
This means $\alpha=\frac{2\pi}{\sqrt{3}}$, $\beta = \frac{\pi}{\sqrt{3}}$
and $\gamma = \frac{3\pi}{2\sqrt{3}}$, and gives the following.

\begin{example}\label{ex:super-strange}
  Consider the Markov generator
\[
    Q \, = \, \myfrac{\pi}{2\sqrt{3}} \begin{pmatrix}
      -9 & 6 & 3 & 0 \\ 2 & -9 & 4 & 3 \\
      3 & 0 & -9 & 6 \\ 4 & 3 & 2 & -9 \end{pmatrix} , 
\]  
which is doubly stochastic, but not symmetric, so 
$\widetilde{Q} = Q^{\trans}\ne Q$. One can check that
\[
   [ Q , Q^{\trans} \ts ] \, = \, \myfrac{4 \pi^2}{3} \begin{pmatrix}
      1 & 0 & -1 & 0 \\ 0 & -1 & 0 & 1 \\
      -1 & 0 & 1 & 0 \\ 0 & 1 & 0 & -1 \end{pmatrix} \, \ne \, \nix \ts ,
\] 
which differs from the situation in Example~\ref{ex:strange}.
On the other hand, one can compute the matrix exponentials of
$Q$ and $Q^{\trans}$, which gives
\[
    M \, = \, \ee^Q \, = \, \ee^{Q^{\trans}} = \, \myfrac{1}{4} \begin{pmatrix}
    1-\ve & 1-\ve & 1+3\ve & 1-\ve \\ 1-\ve & 1-\ve & 1-\ve & 1+3\ve \\
    1+3\ve & 1-\ve & 1-\ve & 1-\ve \\ 1-\ve & 1+3\ve & 1-\ve & 1-\ve
    \end{pmatrix}
\]
with $\ve = \ee^{-2\sqrt{3}\ts \pi} \approx 1.87785 \cdot 10^{-5}$, where 
$\ve = \epsilon^2$ with the constant $\epsilon$ from Example~\ref{ex:strange}
and $\det (M) = \ve^3 = \ee^{-6\sqrt{3}\ts \pi} 
\approx 6.62194 \cdot 10^{-15}$. Here, $M$ is symmetric, hence
reversible for $\bs{p} = \frac{1}{4} \bs{1}$, with spectrum 
$\sigma (M) = \{ 1, \ve, -\ve, -\ve \}$.
We thus have an example of a reversible Markov matrix that is
embeddable via a $\bs{p}\ts\ts$-balanced pair of non-commuting Markov 
generators.

Now, $M^2$ has positive spectrum, and reads
\[
     M^2 = \, \myfrac{1}{4} \begin{pmatrix}
     1+3\ve^2 & 1-\ve^2 & 1-\ve^2 & 1-\ve^2 \\
     1-\ve^2 & 1+3\ve^2 & 1-\ve^2 & 1-\ve^2 \\
     1-\ve^2 & 1-\ve^2 & 1+3\ve^2 & 1-\ve^2 \\
     1-\ve^2 & 1-\ve^2 & 1-\ve^2 & 1+3\ve^2 \end{pmatrix} ,
\]
which is not only embeddable via $2 Q$ and $2 Q^{\trans}$, so
\[
    M^2 \ts = \, \ee^{2 Q}  \ts = \, \ee^{2 Q^{\trans}} \nts = \,
    \ee^Q \ee^{Q^{\trans}} \nts = \, \ee^{Q^{\trans}} \! \ee^Q 
    \ts \ne \, \ee^{Q+Q^{\trans}},
\]   
but also via the principal matrix logarithm of $M^2$, which reads
\[
    L \, = \, \sqrt{3}\ts \pi \begin{pmatrix}
        -3 & 1 & 1 & 1 \\ 1 & -3 & 1 & 1 \\
        1 & 1 & -3 & 1 \\ 1 & 1 & 1 & -3 \end{pmatrix}
        \, \ne \, Q + Q^{\trans} .    
\]
Note that $L$ is a constant-input matrix, which is a special case of
the equal-input matrices considered in Example~\ref{ex:equal-input}.
So, we also have an example of a symmetric (and hence reversible)
Markov matrix with positive spectrum that has three embeddings, one
with a reversible rate matrix, and two other ones via a (non-commuting)
$\bs{p}\ts\ts$-balanced  pair.
\exend
\end{example}

As our examples show, a reversible Markov matrix $M$ can have
non-reversible embeddings as well as multiple embeddings, and
the generators of a balanced pair may or may not commute. This
has to be determined on an individual basis, where multiple
solutions can only occur for sufficiently small values of the 
determinant of $M$.

\subsection{Outlook}
Many practically occurring Markov matrices are reversible, and
if an estimated one fails to be reversible, one can determine the 
nearest reversible one, at least numerically \cite{NW}. 
If a specific family of reversible matrices is investigated, additional
tools might simplify its embeddability question, as we saw for the
equal-input matrices \cite{BS2}. There are certainly other important matrix
classes that could and should be analysed along these lines. 

However, what we have discussed above can only be the first step,
as it considers the time-homogeneous case. More generally,
it is of interest whether any given reversible Markov matrix appears
in a continuous-time Markov flow, as one obtains from the initial value
problem $\dot{X} = X Q$ with $X(0)=\one$ with a time-dependent
family $Q = Q(t)$ of Markov generators. As long as $[ Q(t), Q(s)]=\nix$
holds for all $t,s \geqslant 0$, no new cases emerge. But when
they fail to commute, an entirely new chapter is opened; see
\cite{Joh73,BS4} and references therein for a more detailed discussion. 
This extension certainly deserves further investigations in the future.

\appendix
\section{Poisson process approach to
    Example~\ref{ex:strange}}\label{sec:app}

Consider the matrix
\[
    M (\delta) \, \defeq \, \myfrac{1}{3} \begin{pmatrix}
    1-2\delta & 1+\delta & 1+\delta \\
    1+\delta & 1-2\delta & 1+\delta \\
    1+\delta & 1+\delta & 1-2\delta \end{pmatrix},
\]
for $0<\delta\leqslant \frac{1}{2}$, where $M(\delta)$ is Markov, and
the Markov generators
\[
 Q^{}_{+} (\lambda) \, = \, \lambda \begin{pmatrix}
    -1 & 1 & 0 \\ 0 & -1 & 1 \\ 1 & 0 & -1 \end{pmatrix}
  \quad \text{and} \quad
   Q^{}_{-} (\lambda) \, = \, \lambda \begin{pmatrix}
    -1 & 0 & 1 \\ 1 & -1 & 0 \\ 0 & 1 & -1 \end{pmatrix}
\]
with $\lambda > 0$.  We know from Example~\ref{ex:strange} that
$\lambda = 2 \pi/\sqrt{3}$ gives
$\ee^{ Q^{}_{+} (\lambda) } = \ee^{ Q^{}_{-} (\lambda) } =
M(\epsilon)$ with $\epsilon = \ee^{- \pi \sqrt{3}}$. More generally,
it follows from \cite{BS3} that $M(\delta)$ is embeddable for
$0 < \delta \leqslant \epsilon$, but not embeddable for
$\delta > \epsilon$. The corresponding continuous-time processes
$\bigl( X^{+}_{t} \bigr)_{t\geqslant 0}$ and
$\bigl( X^{-}_{t} \bigr)_{t\geqslant 0}$ possess the transition graphs
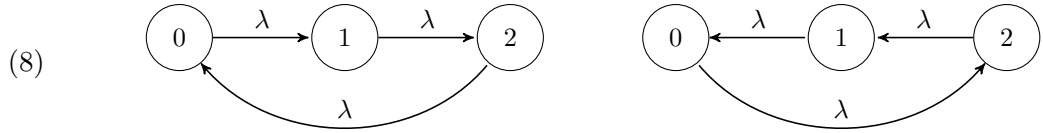
\begin{equation}\label{eq:process}
\raisebox{-30pt}{\small
\begin{tikzpicture}[->,>=stealth',shorten >=1pt,auto,node distance=2.2cm,
                    thin]
  \tikzstyle{every state}=[fill=none,text=black]
  
  \node[state] (A)                    {$0$};
  \node[state] (B) [right of=A]       {$1$};
  \node[state] (C) [right of=B]       {$2$};

  \draw[->, semithick] (A) to node[above] {$\lambda$} (B);
  \draw[->, semithick] (B) to node[above] {$\lambda$} (C);
  \draw[<-, semithick] (A) to[bend right=50] node[above] {$\lambda$} (C);

  \node[state] (D) [right of=C]       {$0$};
  \node[state] (E) [right of=D]       {$1$};
  \node[state] (F) [right of=E]       {$2$};

  \draw[<-, semithick] (D) to node[above] {$\lambda$} (E);
  \draw[->, semithick] (F) to node[above] {$\lambda$} (E);
  \draw[<-, semithick] (F) to[bend left=50] node[above] {$\lambda$} (D);
\end{tikzpicture}}
\end{equation}
and have an interpretation as clockwise or anti-clockwise cycles.

For the further derivation, we need a split of the power series
of the exponential function as $\ee^x = f^{}_{0} (x) + f^{}_{1} (x)
+ f^{}_{3} (x)$ with
\begin{equation}\label{eq:functions}
  f^{}_{\nts\ell} (x) \, = \sum_{m=0}^{\infty}
  \myfrac{x^{3m+\ell}}{(3m+\ell)!} \qquad \text{for}
  \quad \ell \in \{ 0,1,2 \} \ts .
\end{equation}
These functions are well defined and given as follows.

\begin{lemma}\label{lem:functions}
  The power series of the functions\/ $f^{}_{\nts\ell}\ts$ from
  Eq.~\eqref{eq:functions} define entire functions. They are given by
\begin{align*}
  f^{}_{0} (x) \, & = \, \myfrac{1}{3} \Bigl( \ee^x +
      2 \ee^{-x/2} \cos \myfrac{\sqrt{3} \ts x}{2} \Bigr) , \\
  f^{}_{1} (x) \, & = \, \myfrac{1}{3} \Bigl( \ee^x -
      \ee^{-x/2} \Bigl(\cos\myfrac{\sqrt{3} \ts x}{2}
      + \sqrt{3} \sin \myfrac{\sqrt{3} \ts x}{2} \Bigr) \Bigr) , \\
  f^{}_{2} (x) \, & = \, \myfrac{1}{3} \Bigl( \ee^x -
      \ee^{-x/2} \Bigl(\cos\myfrac{\sqrt{3} \ts x}{2}
      - \sqrt{3} \sin \myfrac{\sqrt{3} \ts x}{2} \Bigr) \Bigr) .
\end{align*}
\end{lemma}

\begin{proof}
  The claim on analyticity is standard and justified by the absolute
  convergence of all series involved on all of $\CC$. To compute them,
  define
  $\omega = \ee^{2 \pi \ii /3} = - \frac{1}{2} + \frac{\ii}{2}
  \sqrt{3}$, where one has $\omb = \omega^2$ together with
  $1 + \omega + \omb = 0$. Now, observe that, for
  $k,\ell\in \{0,1,2\}$, one has
\[
  f^{}_{\nts\ell} (\omega^k x) \, = \,
  \omega^{k \ell} f^{}_{\nts\ell} (x) \ts ,
\]
hence $\ee^{\omega^k x} = f^{}_{0} (x) + \omega^k f^{}_{1} (x) +
\omb^k f^{}_{2} (x)$. This allows the following
computations,
\begin{align*}
  f^{}_{0} (x) \, & = \, \myfrac{1}{3} \bigl( \ee^x + \ee^{\omega x}
   + \ee^{\omb x} \bigr) \, = \, \myfrac{1}{3} \Bigl( \ee^x +
     \ee^{-x/2} \bigl( \ee^{\ii \sqrt{3} x /2} + \ee^{- \ii \sqrt{3} x/2}
        \bigr) \Bigr) ,   \\
  f^{}_{1} (x) \, & = \, \myfrac{1}{3} \bigl( \ee^x + \omb \ee^{\omega x}
    + \omega \ee^{\omb x} \bigr) \, = \, \myfrac{1}{3}
        \Bigl( \ee^x - \myfrac{1}{2} \bigl( \ee^{\omega x} +
          \ee^{\omb x} \bigr) - \myfrac{\ii}{2} \sqrt{3}
           \bigl( \ee^{\omega x} \nts - \ee^{\omb x}\bigr) \Bigr) , \\
  f^{}_{2} (x) \, & = \, \myfrac{1}{3} \bigl( \ee^x + \omega \ee^{\omega x}
  + \omb \ee^{\omb x} \bigr) \, = \, \myfrac{1}{3}
        \Bigl( \ee^x - \myfrac{1}{2} \bigl( \ee^{\omega x} +
          \ee^{\omb x} \bigr) + \myfrac{\ii}{2} \sqrt{3}
           \bigl( \ee^{\omega x} \nts - \ee^{\omb x}\bigr) \Bigr) ,
\end{align*}
from which the stated expressions follow via Euler's identity,
$\ee^{\ii y} = \cos (y) + \ii \sin (y)$.  
\end{proof}

We note in passing that, using trigonometric identities, the
expressions for $f^{}_{1}$ and $f^{}_{2}$ could be simplified to
\[
   \cos\myfrac{\sqrt{3} \ts x}{2} \pm \sqrt{3}
   \sin \myfrac{\sqrt{3} \ts x}{2} \, = \, \sin \Bigl(
   \myfrac{\pi}{6} \mp \myfrac{\sqrt{3}\ts x}{2} \Bigr),
\]
but this form would be less suitable for the coming computations. 

Now, let $(S_t )^{}_{t\geqslant 0}$ be an $\NN_0$-valued Poisson
process with parameter (or intensity) $\lambda$, which means that
$S_t$ is Poisson distributed with
$\PP (S_t = k) = \frac{(\lambda t)^k}{k !} \ts \ee^{-\lambda t}$.
Then, the distributions of $ X^{\pm}_{t}$ agree with those of
$\pm S_t \bmod 3 $, and we get
\[
\begin{split}
  p^{}_{0} (\lambda) \, & \defeq \, \PP \bigl( X^{+}_{1} = 0 \bigr) 
  \, = \, \PP \bigl( X^{-}_{1} = 0 \bigr) \, = \,
  \PP \bigl( S^{}_1 \in 3 \ts \NN_0 \bigr)  \\
  & = \, \ee^{-\lambda} \sum_{m=0}^{\infty} \myfrac{\lambda^{3m}}{(3m)!} 
  \, = \, \ee^{-\lambda} f^{}_{0} (\lambda)
  \, = \, \myfrac{1}{3} \Bigl( 1 + 2 \ts \ee^{-3\lambda/2}
     \cos \myfrac{\sqrt{3} \lambda}{2} \Bigr) .
\end{split}
\]
Analogously, we also get
\begin{align*}
  p^{}_{1} (\lambda) \, & = \, \PP \bigl( X^{+}_{1} = 1 \bigr) 
  \, = \,  \PP \bigl( X^{-}_{1} = 2 \bigr) \, = \,
  \PP \bigl( S^{}_1 \in 3 \ts \NN_0 + 1 \bigr)  \\
  & = \, \ee^{-\lambda} f^{}_{1} (\lambda) \, = \,
  \myfrac{1}{3} \Bigl( 1 - \ee^{-3\lambda/2} \Bigl(
  \cos\myfrac{\sqrt{3} \ts x}{2} + \sqrt{3} \sin
  \myfrac{\sqrt{3} \ts x}{2} \Bigr) \Bigr)  \\
\intertext{and}
  p^{}_{2} (\lambda) \, & = \, \PP \bigl( X^{+}_{1} = 2 \bigr) 
  \, = \, \PP \bigl( X^{-}_{1} = 1 \bigr) \, = \,
  \PP \bigl( S^{}_1 \in 3 \ts \NN_0 + 2 \bigr)  \\
  & = \, \ee^{-\lambda} f^{}_{2} (\lambda) \, = \,
  \myfrac{1}{3} \Bigl( 1 - \ee^{-3\lambda/2} \Bigl(
  \cos\myfrac{\sqrt{3} \ts x}{2} - \sqrt{3} \sin
  \myfrac{\sqrt{3} \ts x}{2} \Bigr) \Bigr) .
\end{align*}
Inserting $\lambda = 2 \pi / \sqrt{3}$ gives
$p^{}_{0} (\lambda) = \frac{1}{3} ( 1 - 2 \ts \epsilon)$ and
$p^{}_{1} (\lambda) = p^{}_{2} (\lambda) = \frac{1}{3} (1+\epsilon)$,
which provides an independent derivation of the claim in
Example~\ref{ex:strange}.

In fact, $\lambda = 2 \pi / \sqrt{3}$ is the \emph{smallest}
$\lambda > 0$ with $p^{}_{1} (\lambda) = p^{}_{2} (\lambda)$ and thus
the relation
$\PP \bigl( X^{+}_{1} =\ell \bigr) = \PP \bigl( X^{-}_{1} = \ell
\bigr)$ for $\ell \in \{ 0, 1, 2 \}$, which is equivalent to
$\ee^{Q^{+}(\lambda)}$ being a symmetric Markov matrix. This symmetry
precisely appears for $\lambda_k = 2 (2k+1) \pi/\sqrt{3}$ with
$k\in\NN_0$, then giving $\delta_k = \ee^{-(2k+1)\pi\sqrt{3}}$, with
$\delta^{}_{0} = \epsilon$. This sequence corresponds to the one derived
in \cite{BS3} via a minimisation procedure, and defines a decreasing
sequence with $\delta^{}_k > \delta^{}_{k+1}$ for all $k\in\NN_0$ and
$0$ as its limit point. Recall from \cite[Rem.~3.9]{BS3} that
$M(\delta)$ is embeddable for all $\delta \leqslant \delta^{}_{0}$,
with Markov generators that derive from the $Q^{\pm} (\delta^{}_{0})$
via a simple deformation. Then, when one reaches $\delta_k$, another
new pair of commuting generators emerges, and this process continues. 
Consequently, the number of embeddings of $M(\delta)$ increases without 
bound as $\delta \,\raisebox{2pt}{\tiny $\searrow$}\, 0$
or, equivalently, as the traces of the constant-input matrices
$M(\delta)$ approach $1$.

\section*{Acknowledgements}

This work was supported by the German Research Foundation (DFG),
within the CRC 1283/2 \mbox{(2021 - 317210226)} at Bielefeld
University.

\end{document}